\documentclass{lms}
\usepackage{amsmath}
\usepackage{amsfonts}
\usepackage{graphicx,epsfig}
\usepackage{xcolor,combelow}

\newtheorem{theorem}{Theorem}[section] 
\newtheorem{lemma}[theorem]{Lemma}     

\newnumbered{assertion}{Assertion}    
\newnumbered{conjecture}{Conjecture}  
\newnumbered{definition}{Definition}
\newnumbered{hypothesis}{Hypothesis}
\newnumbered{remark}{Remark}
\newnumbered{note}{Note}
\newnumbered{observation}{Observation}
\newnumbered{problem}{Problem}
\newnumbered{question}{Question}
\newnumbered{algorithm}{Algorithm}
\newnumbered{example}{Example}
\newunnumbered{notation}{Notation} 

\newcommand{\bigO}{\mathcal{O}}
\newcommand{\bigL}{\mathcal{L}}

\title[On the counting function of semiprimes]
 {On the counting function of semiprimes} 

\author{Drago\cb{s} Cri\cb{s}an $\;$and$\;$ Radek Erban}

\classno{11N37 (primary), 11K65, 11Y60 (secondary).
}

\begin{document}
\maketitle

\begin{abstract}
A semiprime is a natural number which can be written as the
product of two primes. The asymptotic behaviour of the
function $\pi_2(x)$, the number of semiprimes less than or
equal to~$x$, is studied. Using a combinatorial argument,
asymptotic series of $\pi_2(x)$ is determined, with all 
the terms explicitly given. An algorithm for the calculation 
of the constants involved in the asymptotic series  
is presented and the constants are computed to
20 significant digits. The errors of the partial sums of the
asymptotic series are investigated. A generalization of
this approach to products of $k$ primes, for $k\geq 3$, is also
proposed.
\end{abstract}

\section{Introduction}

\noindent
For a positive integer $k$ and a positive integer (or real number)
$x$, let $\pi_k(x)$  be the number of integers less than or
equal $x$ which can be written as the product of $k$ prime 
factors. The behaviour of $\pi_k(x)$ has been extensively studied 
during last two centuries, with the main focus on the case $k=1$,
where $\pi_1(x)$ is the prime counting function, denoted 
$\pi(x)$ in the rest of this paper. The prime number theorem 
states that $\pi(x) \sim \text{li}(x)$, where the logarithmic 
integral function 
$\mbox{li}(x)
=
\int_{0}^{x}
\log^{-1} t \; \mbox{d} t
$ 
can be written as an asymptotic expansion 
$
\mbox{li} (x)
\sim 
{\frac {x}{\log x}}
\sum _{n=0}^{\infty }{\frac {n!}{(\log x)^{n}}}
$. 
Bounds on the error term have been established in the literature,
including the recent work of Trudgian~\cite{Trudgian:2016:AET}, 
who proved that, for sufficiently large $x$,
\begin{equation*}
\big| \pi(x) - \mbox{li}(x) \big|
\, \leq  \, 0.2795 \,
\frac{x}{(\log x)^{3/4}} 
\, \exp\left( -\sqrt{\frac{\log x}{6.455}} \right).
\end{equation*}
This implies the existence of constants $d_1$ and $d_2$ such that
\begin{equation}
\label{ineq_noRH}
\big| \pi(x) - \mbox{li}(x) \big|
\, \leq \, d_1 \,
\frac{x}{(\log x)^{3/4}} 
\, \exp\left( - d_2 \sqrt{\log x} \right),
\qquad
\mbox{for all} \; x\geq 2.
\end{equation}
Assuming the Riemann hypothesis, Rosser and
Schoenfeld~\cite{Rosser:1975:SBC,Schoenfeld:1976:SBC} established
even sharper bounds on the error term, including
\begin{equation*}
|\pi(x)-\text{li}(x)| < \frac{\sqrt{x} \log x}{8\pi}
\end{equation*}
for large enough $x$. Other explicit estimates of $\pi(x)$, 
in terms of $x$ and $\log x$ are achievable, as proved by
Axler~\cite{Axler:2006:NBP}.

In this paper, we focus on the case $k=2$, where the numbers 
written as products of two (not necessarily distinct) primes
are called semiprimes. In this case, Ishmukhametov and
Sharifullina~\cite{Ishmukhametov:2014:DSN} recently 
used probabilistic arguments to approximate the behaviour 
of $\pi_2(x)$ as
\begin{equation}
\pi_2(x) 
\approx 
\frac{x \log (\log x)}{\log x} 
+ 
\, 0.265 \, \frac{x}{\log x} 
- 
\, 1.540 \, \frac{x}{(\log x)^2} \, .
\label{ishsha}
\end{equation}
The first term of~(\ref{ishsha}) has already been known to 
Landau~\cite[\S{}56]{Landau:1909:HLV}, with his 
result stated, for general $k \in {\mathbb N}$, as
\begin{equation}
\pi_k(x) 
\sim
\frac{1}{(k-1)!} \, \frac{x \, (\log (\log x))^{k-1}}{\log x} \, .
\label{landaupoi}
\end{equation}
Delange~\cite[Theorem 1]{Delange:1971:FAS} obtained the asymptotic
expansion of $\pi_k(x)$ in the form
\begin{equation}
\pi_k(x) 
\sim 
\frac{x}{\log x}
\,
\sum_{n=0}^{\infty}
\frac{P_{n,k} (\log (\log x))}{(\log x)^n} \, ,
\label{delange}
\end{equation}
where $P_{n,k}$ are polynomials of degree $k-1$, with the
leading coefficient equal to $n!/(k-1)! \, .$ 
Tenenbaum~\cite{Tenenbaum:2015:IAP} proved a similar 
result, giving an expression for the coefficients in the 
polynomial $P_{0,k}$ in terms of the derivatives of 
$
\frac{1}{\Gamma(z+1)} 
\prod_{p} \big( 1 + \frac{z}{p-1} \big) 
\big( 1 - \frac{1}{p} \big)^{\! z}
$ 
evaluated at $z=0$.
Considering $k=2$ in (\ref{delange}), we can write 
an asymptotic expansion for $\pi_2(x)$ as
\begin{equation}
\pi_2 (x) 
\sim
\sum_{n=1}^{\infty} (n-1)! \,
\frac{x \log(\log x)}{(\log x)^n} 
+ 
\sum_{n=1}^{\infty} C_{n-1} \, \frac{x}{(\log x)^n} \, .
\label{aspi2}
\end{equation}
In Theorem~\ref{theoremM}, we prove that $C_0 = M$, 
where $M = 0.261497...$ is the Meissel--Mertens constant 
defined by
\begin{equation}
M 
= 
\lim_{x\rightarrow \infty} 
\Bigg( 
\sum_{p\leq x} \, \frac{1}{p} 
\, - \, \log (\log x) 
\Bigg),
\label{defMMcon}
\end{equation}
where we sum over all primes such that $p \le x.$ 
In Section~\ref{sectionCkcomp},
we calculate the rest of constants $C_n$ 
appearing in equation~(\ref{aspi2}).
They are given in Table~\ref{tablecbk} and obtained 
by the formula
\begin{equation}
C_n 
= 
n! \; \Bigg( 
\sum_{i=0}^{n} \frac{B_i}{i!} - \sum_{i=1}^{n} \frac{1}{i} 
\Bigg)
=
n! \; \Bigg( 
\sum_{i=0}^{n} \frac{B_i}{i!} - H_i
\Bigg),
\label{constCk}
\end{equation}
where $H_i$ is the $i$-th harmonic number, $B_0 = M$ and 
constants $B_i$ are defined using the asymptotic behaviour 
of sums~\cite[\S{}56]{Landau:1909:HLV}
\begin{equation}
\sum_{p\leq x} \frac{(\log p)^i}{p} 
= 
\frac{(\log x)^i}{i} + B_i + \bigO
\! \left( e^{-\sqrt[14]{\log x}} \right),
\qquad
\mbox{for} \; i \in {\mathbb N}.
\label{landaulog}
\end{equation}
Constants $B_i$ are given as limits~(\ref{constBklimit}) in
Section~\ref{sectionCkcomp}, where we present an algorithm
to efficiently calculate them to a desired accuracy. They are 
computed in Table~\ref{tablecbk} to 20 significant digits.
Rosser and Schoenfeld~\cite{Rosser:1962:AFS} prove that 
the error term in (\ref{landaulog}) can be given explicitly 
in terms of an integral, which contains the error terms in 
the prime number theorem. For the case $i=1$ in
equation~(\ref{landaulog}), explicit estimates of this 
sum and, in particular, of the constant $B_1$ involved, 
were recently obtained by Dusart~\cite{Dusart:2018:EES}.

A related arithmetic function, $\Omega(m)$, is defined to 
be the number of prime divisors of $m \in {\mathbb N}$, 
where prime divisors are counted with their multiplicity.
Considering fixed $x$ in equation~(\ref{landaupoi}), we 
can view this approximation of $\pi_k(x)/x$ as the 
probability mass function of the Poisson distribution
with mean $\log(\log(x)).$ Erd\H{o}s and Kac~\cite{Erdos:1940:GLE}
showed that the distribution of $\Omega(x)$ is Gaussian with 
mean $\log(\log(x))$ (see also R\'enyi and
Tur\'an~\cite{Renyi:1958:TEK} and Harper \cite{Harper:2008:TNP} for
generalizations and better bounds). Diaconis~\cite{Diaconis:1976:AEM}
obtained the asymptotic expansions for the average number of 
prime divisors as (see also Finch~\cite[Section 1.4.3]{Finch:2018:MC2})
\begin{equation}
\frac{1}{x} \sum_{m \le x}
\Omega(m)
\, \sim \,
\log (\log x)
+
1.0346538818\dots
+
\sum_{n=1}^\infty
\left(
- 1
+
\sum_{i=0}^{n-1}
\frac{\gamma_i}{i!}
\right)
\frac{(n-1)!}{\log^n x} \, ,
\label{defOmegaLarge}
\end{equation}
where the constants $\gamma_i$ are the Stieltjes constants,
numerically computed in \cite{Bohman:1988:SFD} to 20 
significant digits. An asymptotic series for the 
variance of $\Omega$ have also been 
obtained~\cite[Section 1.4.3]{Finch:2018:MC2}. The Stieltjes
constants $\gamma_i$ are used in Section~\ref{sectionCkcomp}
during our calculation of the values of constants $B_n$ and 
$C_n$, for $n \in {\mathbb N}.$

This paper is organized as follows. Section \ref{k=2} begins
with a counting lemma for expressing the semiprime counting
function $\pi_2$ in terms of the prime counting function $\pi$.
Using this lemma, the main results on the asymptotic 
behaviour of $\pi_2$ are stated and proved in Section~\ref{k=2} 
as Theorem~\ref{theoremM} and Theorem~\ref{theoremgen}.
While Theorem~\ref{theoremM} only gives the first two terms, 
its proof is more coincise than the proof 
of Theorem~\ref{theoremgen}, which 
gives the full asymptotic series of $\pi_2$. The constants 
$C_n$ which appear in this asymptotic series are computed in 
Section~\ref{sectionCkcomp}, where we present an efficient approach to calculate both constants $B_n$ and $C_n$, 
based on the differentiation of the prime zeta function. In Section~\ref{sectionerrorterms}, we
investigate the behaviour of the error terms given by 
the partial sums of the asymptotic series of $\pi_2$. We conclude
with a generalization of the counting argument in 
Section~\ref{secdiscussion}, discussing the extensions of the
presented results to the general case of counting
functions $\pi_k$ for $k\geq 3$. 

\section{Asymptotic behaviour of the counting function
of semiprimes}
\label{k=2}

\noindent 
As in equation (\ref{defMMcon}), we denote primes by 
$p$ and the sums over $p$ shall be understood as sums over
all primes satisfying the given condition.
In the case of summing over primes twice, we 
denote the corresponding prime summation indices 
by $p_1$ and $p_2$. 
We begin with a simple counting
formula~\cite{Ishmukhametov:2014:DSN},
that gives a way of computing $\pi_2 (x)$.

\begin{lemma}
\label{counting_pi2}
For a positive integer $x$, the following holds
\begin{equation}
\pi_2(x) 
= 
\frac{\pi(\sqrt{x}) - \pi(\sqrt{x})^2}{2}
\,
+
\sum_{p\leq \sqrt{x}}
\pi \! \left(\frac{x}{p}\right). 
\label{countformpi2}
\end{equation}
\end{lemma}

\begin{proof} \noindent
By the definition of counting functions
$\pi_2$ and $\pi$, we have
$$
\pi_2(x) 
= 
\sum_{\substack{p_1\leq p_2\\ p_1 p_2 \leq x}} 
\! \! 1
\;
= 
\sum_{p_1 \leq \sqrt{x}} 
\,
\sum_{p_1 \leq p_2 \leq \frac{x}{p_1}} 
\! \! 
1
\;
= 
\sum_{p_1 \leq \sqrt{x}} 
\left( 
\pi \! \left( \frac{x}{p_1} \right) - \pi(p_1) + 1  
\right)
$$
and formula~(\ref{countformpi2}) follows by renaming
$p_1$ to $p$ in the first term and observing that the rest
of the right hand side is the sum of 
all natural numbers from 1 up to $\pi(\sqrt{x})-1.$
\end{proof}

\noindent
Formula~(\ref{countformpi2}) gives an expression of $\pi_2(x)$
in terms of the prime counting function $\pi(x),$ which
can be approximated using the prime number
theorem~\cite{Tenenbaum:2015:IAP} as
\begin{equation}
\pi(x) 
= 
\alpha_n (x) 
+ 
\bigO \! \left( \frac{x}{(\log x)^{n+1}} \right),
\label{PNT}
\end{equation}
where $n \in {\mathbb N}$ and
\begin{equation}
\alpha_n(x)
=
\frac{x}{\log x}
\left( 
\,
\sum_{i=0}^{n-1} 
\frac{i!}{(\log x)^i}
\right).
\label{notationalphak}
\end{equation}
Using Landau~\cite[\S{}56]{Landau:1909:HLV}, we 
can rewrite equation~(\ref{defMMcon}) for any
integer $n \in {\mathbb N}$ as
\begin{equation}
\sum_{p\leq \sqrt{x}} \frac{1}{p}
=
\log (\log x) - \log 2 + M + 
o\left( \frac{1}{(\log x)^n}\right),
\label{sqrtMform}
\end{equation}
where we use the little $o$ asymptotic
notation~\cite{deBruijn:1958:AMA}, as opposed
to the big $\bigO$ asymptotic notation used in 
the prime number theorem~(\ref{PNT}). First we use 
this result to approximate the sum on the right 
hand side of equation~(\ref{countformpi2}). 

\begin{lemma}
\label{pi_to_alphas}
Let $n \in {\mathbb N}$ and $\alpha_n(x)$ be defined 
by~$(\ref{notationalphak})$. Then we have
\begin{equation}
\label{sum_pi_to_sum_alphas}
\sum_{p\leq \sqrt{x}} 
\pi \! \left(\frac{x}{p}\right) 
= 
\sum_{p\leq \sqrt{x}} 
\alpha_n \! \left(\frac{x}{p}\right) 
+ 
\bigO \!
\left( 
\frac{x \log(\log x)}{(\log x)^{n+1}} 
\right).
\end{equation}
\end{lemma}

\begin{proof}
Using equation~(\ref{PNT}), we have
$$
\left|
\sum_{p\leq \sqrt{x}} 
\pi \! \left(\frac{x}{p}\right) 
- 
\sum_{p\leq \sqrt{x}} 
\alpha_n \! \left(\frac{x}{p}\right)
\right|
\, \leq 
\sum_{p\leq \sqrt{x}} 
\left| 
\pi \! \left(\frac{x}{p}\right) 
- 
\alpha_n \! \left(\frac{x}{p}\right)\right|
\, \leq \, 
c \, 
x \sum_{p\leq \sqrt{x}} \frac{1}{p \, (\log x - \log p)^{n+1}}
\;,
$$
where $c>0$ is a constant.
Equation~(\ref{sum_pi_to_sum_alphas}) then follows
by estimating the right hand side by
$$
\frac{c \, x \, 2^{n+1}}{(\log x)^{n+1}} \sum_{p\leq \sqrt{x}} \frac{1}{p}
=
\bigO \! \left( \frac{x \log(\log x)}{(\log x)^{n+1}} \right),
$$
where the last equality follows from
equation~(\ref{sqrtMform}).
\vskip -4mm
\end{proof}

\subsection{The first two terms of the asymptotic
series for $\pi_2(x)$}
\label{k=2_first_terms}

\noindent Using~(\ref{sum_pi_to_sum_alphas}) for $n=1$, 
we obtain
\begin{equation}
\sum_{p\leq \sqrt{x}} \pi \! \left(\frac{x}{p}\right) 
= 
\sum_{p\leq \sqrt{x}} \frac{x}{p \, (\log x - \log p)} 
+ 
\bigO \! \left( \frac{x \log(\log x)}{(\log x)^2 } \right).
\label{kis1sumform}
\end{equation}
Using Landau~\cite[\S{}56]{Landau:1909:HLV}, we can
rewrite equation~(\ref{landaulog}) 
for any integers $i \in {\mathbb N}$ 
and $n \in {\mathbb N}$ as
\begin{equation}
\sum_{p\leq \sqrt{x}} 
\frac{(\log p)^i}{p} 
=
\frac{(\log x)^i}{i \, 2^i} 
+ 
B_i 
+ 
o
\left( \frac{1}{(\log x)^n}\right).
\label{constBk}
\end{equation}
We will use this to prove the first theorem of this section.

\begin{theorem}
\label{theoremM}
Let $M$ be the Meissel--Mertens constant defined by
$(\ref{defMMcon})$. Then
\begin{equation}
\label{pi_2_first_two_terms}
\pi_2 (x) 
= 
\frac{x \log (\log x)}{\log x} 
+ 
M {\hskip 0.1mm} \frac{x}{\log x} 
+ 
o\left( \frac{x}{\log x}\right).
\end{equation}
\end{theorem}

\begin{proof}
Using equations~(\ref{countformpi2}),
(\ref{PNT}) for $n=1$, and
(\ref{kis1sumform}), we obtain
\begin{equation}
\label{pi_2_first_two_terms_computations}
\pi_2(x) 
=
\sum_{p\leq \sqrt{x}} 
\pi \! \left(\frac{x}{p}\right) 
+ o \left( \frac{x}{\log x}\right)
= 
\frac{x}{\log x}
\sum_{p\leq \sqrt{x}} \frac{\log x}{p \, (\log x - \log p)} 
+ o\left( \frac{x}{\log x}\right).
\end{equation}
We have the following identity
$$
\frac{\log x}{p \, (\log x - \log p)}
=
\frac{1}{p} 
+ 
\frac{\log p}{\log x}
\left(
\frac{\log x}{p \, (\log x - \log p)}
\right).
$$
Substituting the left hand side into the right hand side,
we obtain, for any natural number $n \in {\mathbb N}$,
that
$$
\frac{\log x}{p \, (\log x - \log p)}
=
\frac{1}{p} 
\sum_{i=0}^n 
\left( 
\frac{\log p}{\log x} 
\right)^{\!\!i}
\, + \, 
\left(
\frac{\log p}{\log x}
\right)^{\!\!n+1}
\!
\left(
\frac{\log x}{p \, (\log x - \log p)}
\right).
$$
Summing over all primes $p \le \sqrt{x}\,$, we get
\begin{equation*}
\sum_{p\leq \sqrt{x}} 
\frac{\log x}{p \, (\log x - \log p)}
=
\sum_{i=0}^n \left( \frac{1}{(\log x)^{i}} 
\sum_{p\leq \sqrt{x}} \frac{(\log p)^i}{p}
\right) 
+
\frac{1}{(\log x)^{n}} \sum_{p\leq \sqrt{x}} 
\frac{(\log p)^{n+1}}{p \, (\log x - \log p)}.
\end{equation*}
Each term on the right hand side can be evaluated
using equations~(\ref{sqrtMform}) and~(\ref{constBk}),
with $o(1)$ accuracy, as
\begin{equation}
\sum_{p\leq \sqrt{x}} 
\frac{\log x}{p \, (\log x - \log p)}
=
\log(\log x) 
- \log 2 
+ 
M
+ 
\sum_{i=1}^n
\frac{1}{i \, 2^i}
+
f(n) 
+
o(1) {\hskip 0.2mm} ,
\label{intermform}
\end{equation}
where $f(n)$ is a decreasing function of $n$ satisfying
$f(n) \to 0$ as $n \to \infty$. This can be deduced by
applying equation~(\ref{constBk}) to the error term
$$
\frac{1}{(\log x)^{n}} 
\sum_{p\leq \sqrt{x}} 
\frac{(\log p)^{n+1}}{p \, (\log x - \log p)}
\leq 
\frac{2}{(\log x)^{n+1}} 
\sum_{p\leq \sqrt{x}} \frac{(\log p)^{n+1}}{p}
=
\frac{1}{(n+1) \, 2^n}
+
o(1).
$$
Since we have
$$
- 
\log 2 
+ 
\sum_{i=1}^n
\frac{1}{i \, 2^i}
+
f(n) 
\to
0,
\qquad
\mbox{as}
\quad
n \to \infty,
$$
equation~(\ref{intermform}) implies that
$$
\sum_{p\leq \sqrt{x}} 
\frac{\log x}{p \, (\log x - \log p)}
=
\log(\log x) 
+ 
M
+
o(1).
$$
Substituting into
equation~(\ref{pi_2_first_two_terms_computations}),
we obtain formula~(\ref{pi_2_first_two_terms}).
\end{proof}

\subsection{Asymptotic series for the counting
function of semiprimes}
\label{k=2_complete_series}

\noindent To derive formulas for all terms in the asymptotic
series of $\pi_2$, we first define an auxiliary sequence of
numbers $q_n$ for $n \in {\mathbb N}$ by
\begin{equation}
q_n
=
\sum_{i=1}^{n-1}
\frac{2^{i}-1}{i},
\quad
\mbox{for}
\; \; n \ge 2,
\qquad
\mbox{and}
\qquad
q_1=0 \, .
\label{defqn}
\end{equation}
Then $q_n$ is an increasing sequence of rational numbers
with the first few terms given
as $q_1 = 0,$ $q_2 = 1$, $q_3 = 5/2$, $q_4 = 29/6$,
$q_5 = 103/12$ and $q_6 = 887/60$, which satisfies 
the following identity.

\begin{lemma}
\label{lemmaqn}
Let $n \in {\mathbb N}$ and let $q_n$ be given 
by equation~$(\ref{defqn})$. Then we have
\begin{equation}
\label{qnidentity}
\sum_{i=1}^{\infty} 
\binom{n+i-1}{n-1} \, \frac{1}{i \, 2^i} 
= 
q_n + \log 2 \, . 
\end{equation}
\end{lemma}

\begin{proof}
Considering the binomial series
\begin{equation*}
(1-t)^{-n} 
= 
1 + t \, \sum_{i=0}^{\infty} 
\binom{n+i}{n-1} \, t^i,
\qquad
\mbox{for} \; t\in(-1,1),
\end{equation*}
we can rewrite it as
$$
\sum_{i=0}^{\infty} \binom{n+i}{n-1} \, t^i 
= 
\frac{(1-t)^{-n}-1}{t}
= 
\sum_{i=1}^n
(1-t)^{-i}.
$$
Integrating, we get
\begin{equation*}
\sum_{i=0}^{\infty} 
\binom{n+i}{n-1} \, \frac{t^{i+1}}{i+1} 
= 
-
\log(1-t)
+
\sum_{i=2}^n
\frac{(1-t)^{-i+1}-1}{i-1},
\end{equation*}
which holds for $t\in(-1,1)$.
Substituting $t=1/2$, we obtain (\ref{qnidentity}).
\end{proof}

\noindent
We will use Lemma~\ref{lemmaqn} in the proof of the following
theorem, giving the asymptotic series for the semiprime
counting function $\pi_2(x).$

\begin{theorem}
\label{theoremgen}
The constants $C_n$ appearing in the asymptotic 
expansion~$(\ref{aspi2})$ are given 
by equation~$(\ref{constCk})$ for $n \in {\mathbb N}$
and as $C_0 = B_0 = M$ for $n=0.$ 
\end{theorem}

\begin{proof}
\noindent
The case $n=0$ is studied in Theorem~\ref{theoremM}, which
states that $C_0 = M.$ To derive equation~(\ref{constCk}),
we again use formula~(\ref{countformpi2}) from 
Lemma~\ref{counting_pi2} and approximate each term
using the prime number theorem~(\ref{PNT}). We
need to analyze sums of the form 
\begin{equation}
S_n(x)
=
\sum_{p\leq \sqrt{x}} 
\frac{(\log x)^n}{p \, (\log x - \log p)^n} 
=
\sum_{p\leq \sqrt{x}} 
\frac{1}{p}
\left( 1-\frac{\log p}{\log x}\right)^{\!\!-n} \!\! .
\label{defSnx}
\end{equation}
Using the binomial series on the right hand side, we get
\begin{equation}
S_n(x) 
\, = \, 
\sum_{p\leq \sqrt{x}} 
\frac{1}{p} 
\sum_{i=0}^{\infty} \binom{n+i-1}{n-1} 
\left( \frac{\log p}{\log x}\right)^{\!\!i}
\, = \,
\sum_{i=0}^{\infty} \binom{n+i-1}{n-1} 
\frac{1}{(\log x)^i} \sum_{p\leq \sqrt{x}} 
\frac{(\log p)^i}{p} \, . \;\;
\label{S_k_formula0}
\end{equation}
Substituting $x^2$ for $x$, we obtain
\begin{equation}
\label{S_k_formula1}
S_n(x^2) 
= 
\sum_{i=0}^{\infty} \binom{n+i-1}{n-1} 
\frac{1}{2^i (\log x)^i} \sum_{p\leq x} 
\frac{(\log p)^i}{p} \, .
\end{equation}
To estimate the sums over primes on the right hand side,
we apply the result of Rosser and
Schoenfeld~\cite[equation (2.26)]{Rosser:1962:AFS}, 
which can be formulated as
$$
\sum_{p\leq x} \frac{1}{p} = \log (\log x) + \bigL_0(x),
\qquad
\sum_{p\leq x} 
\frac{(\log p)^i}{p} = \frac{(\log x)^i}{i} + \bigL_i(x),
$$
where the error terms $\bigL_i(x)$ are defined by
\begin{eqnarray*}
\bigL_0(x) &=&
-
\log(\log 2)
+ 
\frac{\mbox{li}(2)}{2} 
+ 
\frac{\pi(x) - \mbox{li}(x)}{x}
+ 
\int_2^{x} 
\frac{\pi(y) - \mbox{li}(y)}{y^2} \,
\mbox{d}y,
\\
\bigL_i(x) &=& 
-
\frac{(\log 2)^i}{i} 
+ 
\frac{(\log 2)^i \, \mbox{li}(2)}{2} 
+ \frac{(\log x)^i}{x} \big( \pi(x) - \mbox{li}(x) \big) \\
&&
+ \, \int_2^{x} 
\frac{(\log y-i) \,
(\log y)^{i-1}}{y^2} 
\big( \pi(y) - \mbox{li}(y) \big) \,
\mbox{d}y,
\qquad \mbox{for} \; i \in {\mathbb N}.
\end{eqnarray*}
Using this notation and identity~(\ref{qnidentity})
in Lemma~\ref{lemmaqn}, we rewrite
equation~(\ref{S_k_formula1}) as
\begin{equation}
\label{S_k_formula2}
S_n(x^2) 
= 
\log (\log x)
+
q_n 
+
\log(2)
+
\sum_{i=0}^{\infty} \binom{n+i-1}{n-1} 
\frac{1}{2^i} 
\frac{\bigL_i(x)}{(\log x)^i} \, .
\end{equation}
Using the inequality (\ref{ineq_noRH}), there exist 
constants $d_1$ and $d_2$ such that
\begin{eqnarray}
\left|\frac{\bigL_i(x)}{(\log x)^i} \right|
&<& 
\frac{2}{(\log x)^i} 
+ 
\frac{\left| \pi(x) - \mbox{li}(x) \right|}{x} 
+ 
\frac{1}{(\log x)^i} 
\int_2^{x} 
\frac{|i - \log y| \, (\log y)^{i-1}}{y^2} 
\big| \pi(y) - \mbox{li}(y) \big| \, \mbox{d}y
\nonumber
\\
&\leq& 
\frac{2}{(\log x)^i} 
+ 
d_1 \,
\frac{
\exp\left( 
- d_2 \sqrt{\log x} \right)
}{(\log x)^{3/4}}
+ 
\frac{i 
\, I(i-1,x)
+
I(i,x)}{(\log x)^i} \, ,
\label{estimateLi}
\end{eqnarray}
where we define 
(note that we allow the second argument to be $\infty$ in this definition):
$$
I(i,x) 
= 
d_1
\int_2^{x} 
(\log y)^{i-3/4} \,  
\frac{
\exp\left( -d_2 \sqrt{\log y} 
\right)}{y} \, \mbox{d}y \, .
$$
Choose $\ell \in {\mathbb N}$. Our goal is to use
(\ref{estimateLi}) to estimate the rate of 
convergence of the sum on the right hand of 
equation~(\ref{S_k_formula2}). To do this we
first observe that, for $i \ge \ell$,
we have
\begin{equation}
\frac{
I(i,x)}{(\log x)^i}
\le
\frac{d_1}{(\log x)^\ell}
\int_2^{x} 
(\log y)^{\ell-3/4} \,  
\frac{
\exp\left( -d_2 \sqrt{\log y} 
\right)}{y} \, \mbox{d}y
\le
\frac{I(\ell,\infty)}{(\log x)^\ell} \,.
\label{Iiellineq}
\end{equation}
Using inequalities~(\ref{estimateLi}) 
and~(\ref{Iiellineq}) and assuming $\log(x) \ge 1$, 
we can estimate the remainder 
of the series on the right hand of 
equation~(\ref{S_k_formula2}) as
\begin{eqnarray*}
&& \left|
\sum_{i=\ell+1}^{\infty} 
\binom{n+i-1}{n-1} 
\frac{1}{2^i} 
\frac{\bigL_i(x)}{(\log x)^i}
\right|
\; \leq \; 
\left(
\frac{2}{(\log x)^{\ell+1}} 
+ 
d_1 \,
\frac{
\exp\left( 
- d_2 \sqrt{\log x} \right)
}{(\log x)^{3/4}}
\right)
\sum_{i=\ell+1}^{\infty} 
\binom{n+i-1}{n-1} 
\frac{1}{2^{i}}
\\
&& 
\hskip 2cm
+ 
\;
\frac{
\, I(\ell,\infty)
}{(\log x)^{\ell+1}}
\sum_{i=\ell+1}^{\infty} 
\binom{n+i-1}{n-1} 
\frac{i}{2^i} 
\;
+ 
\;
\frac{I(\ell+1,\infty)}{(\log x)^{\ell+1}}
\sum_{i=\ell+1}^{\infty} 
\binom{n+i-1}{n-1} 
\frac{1}{2^i} \,. 
\end{eqnarray*}
Since all three sums on the right hand side converge
independently of $x$, we deduce that the
remainder is of the order
$\bigO \big( (\log x)^{-(\ell+1)}) \big)$.
Therefore, equation~(\ref{S_k_formula2}) becomes
\begin{equation*}
S_n(x^2) 
= 
\log (\log x)
+
q_n 
+
\log(2)
+
\sum_{i=0}^{\ell} \binom{n+i-1}{n-1} 
\frac{1}{2^i} 
\frac{\bigL_i(x)}{(\log x)^i}    
+ 
\bigO\left( \frac{1}{(\log x)^{\ell+1}} \right).
\end{equation*}
This means that an asymptotic expansion of 
$S_n(x^2)$ in terms of negative powers of 
$\log x$ is given by the sum of the asymptotic 
series of terms in equation~(\ref{S_k_formula1}).
The same is true for $S_n(x)$ in 
equation~(\ref{S_k_formula0}). Thus, using
equations~(\ref{sqrtMform}), (\ref{constBk}),
(\ref{qnidentity}) and~(\ref{S_k_formula0}),
we obtain
\begin{eqnarray}
S_n(x) 
&=&
\sum_{i=0}^{\ell} \binom{n+i-1}{n-1} 
\frac{1}{(\log x)^i} \sum_{p\leq \sqrt{x}} 
\frac{(\log p)^i}{p}
+ 
\bigO\left( \frac{1}{(\log x)^{\ell+1}} \right)
\nonumber
\\
&=& 
\log(\log x) 
+
M 
+ 
q_n
+
\sum_{i=1}^{\ell} 
\binom{n+i-1}{n-1} \, \frac{B_i}{(\log x)^i} 
+ 
\bigO\left( \frac{1}{(\log x)^{\ell+1}} \right).
\label{Siasympt}
\end{eqnarray}
Using equations~(\ref{notationalphak})
and~(\ref{sum_pi_to_sum_alphas}), 
we have
\begin{equation*}
\sum_{p\leq \sqrt{x}} 
\pi \! \left(\frac{x}{p}\right) 
+ 
o
\left(\! 
\frac{x}{(\log x)^{\ell}} 
\!\right)
=
\sum_{n=1}^{\ell}
\sum_{p\leq \sqrt{x}} 
\frac{x \, (n-1)!}{p \, (\log x - \log p)^{n}}
=
\sum_{n=1}^{\ell} (n-1)! \,
\frac{x \, S_n(x)}{(\log x)^n} \,,
\end{equation*}
where we used the definition~(\ref{defSnx}) 
of $S_n(x)$ to get the second equality.
Using equation~(\ref{Siasympt}) and notation
$B_0 = M$, we obtain
\begin{equation*}
\sum_{p\leq \sqrt{x}} 
\pi \! \left(\frac{x}{p}\right) 
=
\sum_{n=1}^{\ell} (n-1)! \, 
\frac{x \log(\log x)}{(\log x)^n} 
+ 
\sum_{n=1}^{\ell} (n-1)!
\left(
q_n
+
\sum_{i=0}^{n-1} 
\frac{B_{i}}{i!}
\right) 
\frac{x}{(\log x)^n} 
+
o
\left(\! 
\frac{x}{(\log x)^{\ell}} 
\!\right).
\end{equation*}
Thus, using formula~(\ref{countformpi2}) and 
the prime number theorem~(\ref{PNT}), 
we obtain the asymptotic expansion~(\ref{aspi2}), 
where we have
$$
C_{n} 
=
n!
\left(
q_{n+1}
+
\sum_{i=0}^{n} 
\frac{B_{i}}{i!}
\right)
-
2^{n} \sum_{i=1}^{n} (i-1)! \, (n-i)! \; .
$$
This can be further simplified by using 
definition~(\ref{defqn}) of $q_n$. We get
$$
C_n
= 
n! 
\left( 
\sum_{i=1}^{n} \frac{2^i-1}{i} 
+
\sum_{i=0}^{n} \frac{B_i}{i!} 
-
\frac{2^{n}}{n} 
\sum_{i=0}^{n-1} 
\frac{1}{\binom{n-1}{i}}
\right) 
= 
n! 
\left( 
\sum_{i=1}^{n} \frac{2^i-1}{i} 
+
\sum_{i=0}^{n} \frac{B_i}{i!} 
-
\sum_{i=1}^{n} \frac{2^{i}}{i}
\right).
$$
Subtracting the first and the third sum, we obtain
(\ref{constCk}).
\end{proof}

\section{Computing the constants}
\label{sectionCkcomp}

\noindent
In this section, we use a fast converging series 
to determine the values of the constants $B_n$ and, 
as a result, of the constants $C_n$ given by equation
(\ref{constCk}). The first constant, $B_0=C_0=M$, 
is the well-studied Meissel--Mertens constant, 
so we will focus on constants $B_n$ in the 
case $n \geq 1$. They have been 
defined by equations~(\ref{landaulog}) 
or~(\ref{constBk}), which can be rewritten as
\begin{equation}
B_n
= 
\lim_{x\rightarrow \infty} 
\left(
\sum_{p\leq \sqrt{x}} \frac{(\log p)^n}{p} 
- \frac{(\log x)^n}{n \, 2^n}
\right)
= 
\lim_{x\rightarrow \infty} 
\left(
\sum_{p\leq x} \frac{(\log p)^n}{p} 
- \frac{(\log x)^n}{n}\right).
\label{constBklimit}
\end{equation}
To derive a formula for evaluating $B_n$ on a computer, 
we use the prime zeta function~\cite{Froberg:1968:PZF}
defined by
\begin{equation}
P(s) = \sum_{p} \frac{1}{p^s} \;,
\qquad \quad
\mbox{for} \; \; s \in (1,\infty).
\label{primezeta}
\end{equation}
Differentiating equation~(\ref{primezeta}), we get 
the formula for the $n$-th derivative of the
prime zeta function as
\begin{equation}
P^{(n)}(s)
=
(-1)^n \, 
\sum_{p} \frac{(\log p)^n}{p^{s}}
=
(-1)^n \, 
\sum_{p \le x} \frac{(\log p)^n}{p^{s}}
+
(-1)^n \, 
\sum_{p > x} \frac{(\log p)^n}{p^{s}}.
\label{derprimezeta1}
\end{equation}
The prime zeta function $P(s)$ can also be related to 
the Riemann zeta function $\zeta(s)$ through the
formula~\cite{Froberg:1968:PZF}
\begin{equation*}
P(s)
=
\sum_{i=1}^{\infty}
\mu(i) 
\,
\frac{\log \big( \zeta(i \, s) \big)}{i} 
\, ,
\qquad \quad
\mbox{for} \; \; s \in (1,\infty),
\end{equation*}
where $\mu(n)$ is the M\"obius function. Taking the
derivative of order $n$ of this expression, 
we obtain
\begin{equation*}
P^{(n)}(s)
=
\sum_{i=1}^{\infty} 
\mu(i) 
\, 
i^{n-1} 
\,
\left(\frac{\zeta'}{\zeta}\right)^{\!\!(n-1)} 
\!\!\!\!\!\!\!\!(i \, s) \, .
\end{equation*}
Substituting into (\ref{derprimezeta1}), we obtain
\begin{equation}
\sum_{p \le x} \frac{(\log p)^n}{p^{s}}
=
(-1)^n \, 
\sum_{i=1}^{\infty} 
\mu(i) 
\, 
i^{n-1} 
\,
\left(\frac{\zeta'}{\zeta}\right)^{\!\!(n-1)} 
\!\!\!\!\!\!\!\!(i \, s)
\; - \; 
\sum_{p > x} \frac{(\log p)^n}{p^{s}} \, .
\label{fromeqder}
\end{equation}
Using integration by parts, we obtain
$$
\int_{1}^{\infty} \frac{(\log u)^{n-1}}{u^{s}} 
\; \mbox{d}u 
=
\frac{(n-1)!}{(s-1)^n} 
=
-
\left(
\frac{(-1)^n}{s-1}
\right)^{\!\!(n-1)}
\!\! .
$$
Thus, the second sum on the right hand side of
equation~(\ref{fromeqder}) can be 
approximated by
$$
\sum_{p > x} \frac{(\log p)^n}{p^{s}}
=
\int_{x}^{\infty} \frac{(\log u)^{n-1}}{u^{s}} 
\; \mbox{d}u 
+
o(1)
=
-
\left(
\frac{(-1)^n}{s-1}
\right)^{\!\!(n-1)}
\!\! - \,
\int_{1}^{x} \frac{(\log u)^{n-1}}{u^{s}} 
\; \mbox{d}u 
+
o(1) \, .
$$
Substituting into equation~(\ref{fromeqder}), we
get
\begin{equation*}
\sum_{p \le x} \frac{(\log p)^n}{p^{s}}
-
\int_{1}^{x} \frac{(\log u)^{n-1}}{u^{s}} 
\; \mbox{d}u 
=
(-1)^n \, 
\sum_{i=1}^{\infty} 
\mu(i) 
\, 
i^{n-1} 
\,
\left(\frac{\zeta'}{\zeta}\right)^{\!\!(n-1)} 
\!\!\!\!\!\!\!\!(i \, s)
\, + \,
\left(
\frac{(-1)^n}{s-1}
\right)^{\!\!(n-1)}
\!\!\!\! 
+ \,
o(1) \, .
\end{equation*}
Taking the limit as $s\rightarrow 1$ and substituting
into equation~(\ref{constBklimit}), we obtain
\begin{equation}
B_n
=
(-1)^n 
\sum_{i=2}^{\infty} 
\mu(i) 
\, 
i^{n-1} 
\left(\frac{\zeta'}{\zeta}\right)^{\!\!(n-1)}
\!\!\!\!\!\!\!\!(i)
\; + \; 
(-1)^n
\lim_{s\rightarrow 1} 
\left( 
\frac{\zeta'(s)}{\zeta(s)}
+
\frac{1}{s-1}
\right)^{\!\!(n-1)} \!\!\!\! ,
\label{Bkform}
\end{equation}
where the first term on the right hand side is a quickly
converging series and the limit in the second term
can be evaluated using the Laurent expansion of $\zeta(s)$
around $s=1$. This is given by
$$
\zeta (s) 
= 
\frac{1}{s-1} 
+ 
\sum_{n=0}^{\infty} 
\frac{(-1)^n}{n!} 
\, 
\gamma_n 
\,
(s-1)^n,
$$
where the Stieltjes constants $\gamma_n$ are computed 
to $20$ significant digits in~\cite{Bohman:1988:SFD}. 
Then, the Laurent expansion of the logarithmic
derivative~\cite{Broughan:2009:ERF} of the Riemann 
zeta function is
\begin{eqnarray}
\frac{\zeta'(s)}{\zeta(s)} 
&=& 
-\frac{1}{s-1} 
+ \gamma_0
+ (-2 \gamma_1 - \gamma_0 ^2) (s-1)
+ \left( \frac{3}{2} \gamma_2 + 3 \gamma_0 \gamma_1 
+ \gamma_0 ^3 \right) (s-1)^2
\nonumber
\\
&&
{\hskip 1cm}
+ 
\;
\left( -\frac{2}{3} \gamma_3 - 2\gamma_0 \gamma_2 
-2 \gamma_1 ^2 - 4\gamma_0 ^2 \gamma_1 - \gamma_0 ^4 
\right) (s-1)^3 \, + \, \ldots.
\label{laurent}
\end{eqnarray}
Constants $B_n$ computed by formula~(\ref{Bkform}) using
the Laurent series~(\ref{laurent}) are presented in
Table~\ref{tablecbk}. Once we know constants 
$B_n$ to the desired accuracy, we can use equation~(\ref{constCk}) to calculate constants
$C_n$. They are also presented in Table~\ref{tablecbk}
to 20 significant digits.

\begin{table}[t]
{\normalsize
\begin{center}
\begin{tabular}{|r|l|l|l|}
\hline
$n$ & {\hskip 17mm}$B_n$ & {\hskip 19mm}$C_n$ \\
\hline
{\hskip 3mm}$0$\rule{0pt}{3.75mm} 
& {\hskip 3mm}$0.26149721284764278375$
& {\hskip 3mm}$0.26149721284764278375$
\\
{\hskip 3mm}$1$ 
& $-1.3325822757332208817$ 
& $-2.0710850628855780875$
\\
{\hskip 3mm}$2$ 
& $-2.5551076154464547041$ 
& $-7.6972777412176108802$
\\
{\hskip 3mm}$3$ 
& $-10.253827096911327612$ 
& $-35.345660320564161516$
\\
{\hskip 3mm}$4$ 
& $-59.332397971808450296$ 
& $-206.71503925406509339$
\\
{\hskip 3mm}$5$ 
& $-453.62459086132753356$ 
& $-1.5111997871316530251  \times 10^3$
\\
{\hskip 3mm}$6$ 
& $-4.3591249600559955673 \times 10^3$ 
& $-1.3546323682845914021 \times 10^4$
\\
{\hskip 3mm}$7$ 
& $-5.0684840978914262902 \times 10^4$ 
& $-1.4622910675883565523 \times 10^5$
\\
{\hskip 3mm}$8$ 
& $-6.9270677393697978276 \times 10^5$ 
& $-1.8675796280076650637 \times 10^6$
\\
{\hskip 3mm}$9$ 
& $-1.0884508606344556845 \times 10^7$ 
& $-2.7733045258413542557 \times 10^7$
\\
$10$ 
& $-1.9329009099289751454 \times 10^8$ 
& $-4.7098342357703294361 \times 10^8$
\\
\hline
\end{tabular}
\end{center}
}
\caption{Table of constants $B_n$ and $C_n$, for
$n=0,1,2,\dots,10$, defined by equations~$(\ref{constBk})$
and~$(\ref{constCk})$, which appear in the asymptotic
expansion of $\pi_2(x)$. The values of constants
$B_n$ are computed by formula~$(\ref{Bkform})$ using
the Laurent series~$(\ref{laurent})$. The values 
of constants $C_n$ are computed by
equation~$(\ref{constCk})$.
\label{tablecbk}}
\end{table}

\section{Computational results: behaviour 
of error terms}

\label{sectionerrorterms}

\begin{figure}
\leftline{\hskip 2mm (a) \hskip 6.95cm (b)}
\centerline{{\hskip 4mm}
\includegraphics[height=5.5cm]{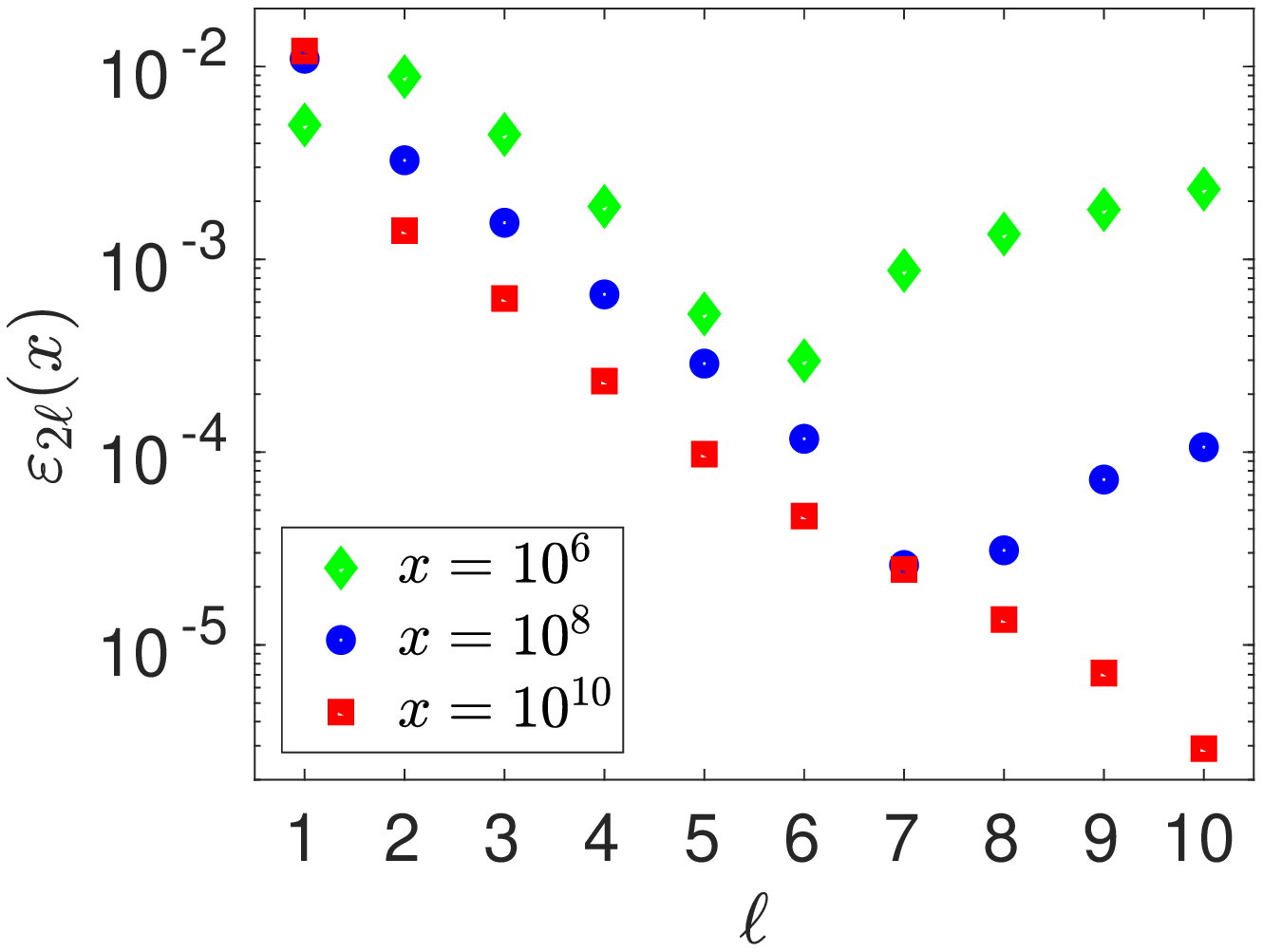}
\includegraphics[height=5.5cm]{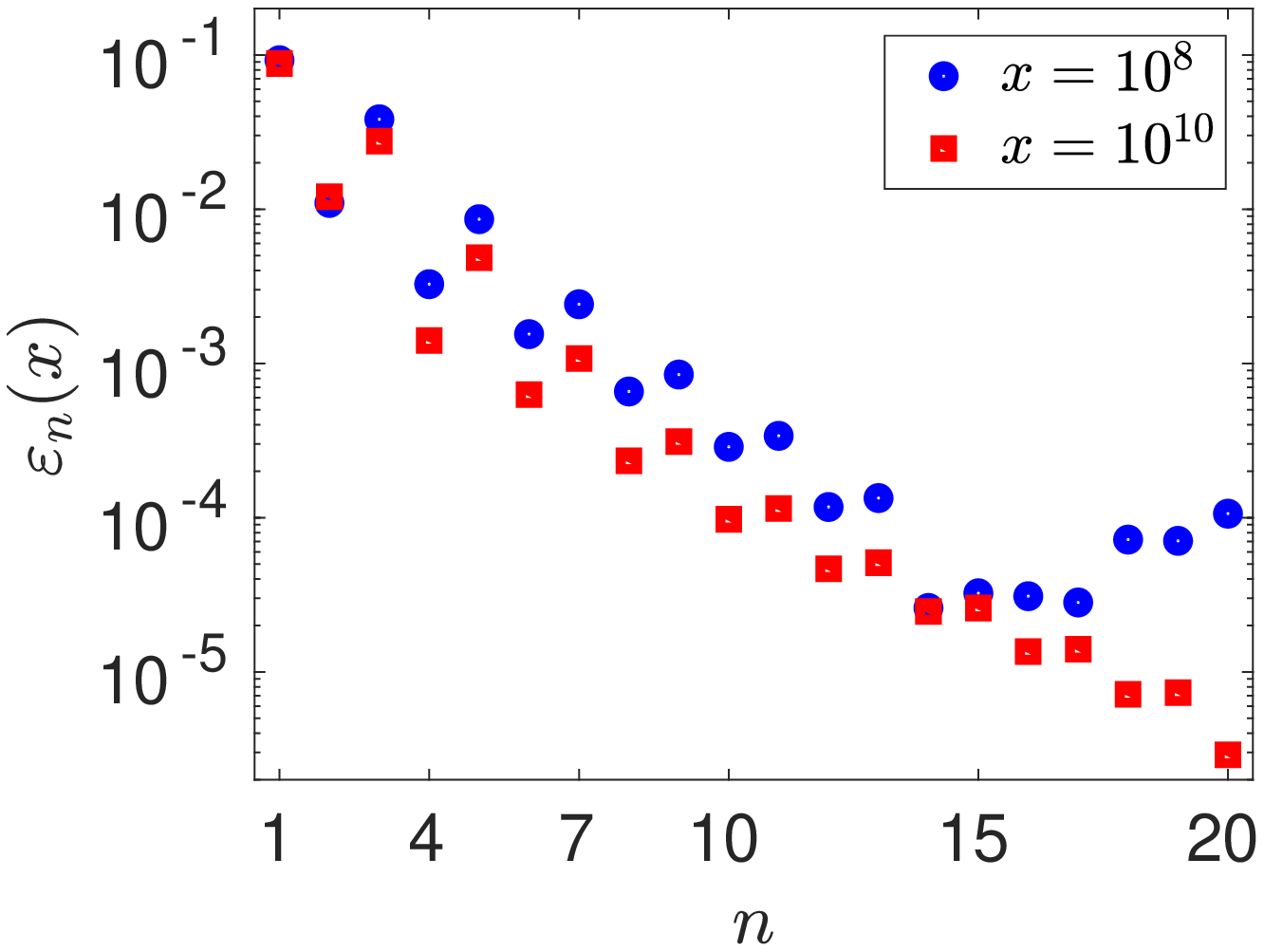}}
\caption{(a) 
{\it Relative errors calculated
by~$(\ref{epsevenll})$ for $x=10^6$, $x=10^8$ and
$x=10^{10}$.} \hfill \break
(b) {\it Relative errors calculated
by~$(\ref{defvarepsk})$ for $x=10^8$ and
$x=10^{10}$. The relative error of 
Landau's approximation~$(\ref{landauapprox})$ 
is plotted as $\varepsilon_1(x)$. The relative error
of the approximation given in 
Theorem~$\ref{theoremM}$ is plotted 
as $\varepsilon_2(x)$.}
}
\label{figerror}
\end{figure}

\noindent
In this section, we illustrate the accuracy of
the asymptotic series (\ref{aspi2}) by calculating 
its error terms at each order. Since (\ref{aspi2})
is a sum of two formal asymptotic series, we have
two ways to define its errors. First, we can
truncate both sums after the same number, $\ell$, 
of terms to define the relative error
\begin{equation}
\varepsilon_{2 \ell}(x)
=
\frac{1}{\pi_2(x)}
\left| 
\left(\sum_{n=1}^{\ell} (n-1)! \,
\frac{x \log(\log x)}{(\log x)^n} 
+ \sum_{n=1}^{\ell} C_{n-1} \frac{x}{(\log x)^n}
\right) 
- \pi_2(x)
\right|,
\quad
\mbox{for} \; \ell \in {\mathbb N}, \;\;
\label{epsevenll}
\end{equation}
which is plotted in Figure~\ref{figerror}(a) for $x=10^6$,
$x=10^8$ and $x=10^{10}$. Combining both sums in the
asymptotic expansion~(\ref{aspi2}) into one, we can 
write it as
$$
\pi_2(x)
\sim
\frac{x \log(\log x)}{\log x}
+ 
C_0 \, \frac{x}{\log x} 
+
\frac{x \log(\log x)}{(\log x)^2} 
+ 
C_1 \, \frac{x}{(\log x)^2}
+
\frac{2 \, x \log(\log x)}{(\log x)^3}
+
\dots,
$$
so we can define the $n$-th approximation, $a_n(x)$,
by
\begin{eqnarray}
a_1(x)
&=&
\frac{x \log(\log x)}{\log x} \, ,
\label{landauapprox}
\\
a_{2 \ell}(x)
&=&
\sum_{i=1}^{\ell} (i-1)! \,
\frac{x \log(\log x)}{(\log x)^i} 
+ \sum_{i=1}^{\ell} C_{i-1} \frac{x}{(\log x)^i} \, ,
\qquad \mbox{for} \; \ell = 1, \, 2, \, 3, \, \dots,
\label{evena2l}
\\
a_{2 \ell+1}(x)
&=&
\sum_{i=1}^{\ell + 1} (i-1)! \,
\frac{x \log(\log x)}{(\log x)^i} 
+ 
\sum_{i=1}^{\ell} C_{i-1} \frac{x}{(\log x)^i} \, ,
\qquad \mbox{for} \; \ell = 1, \, 2, \, 3, \, \dots,
\label{odd2lp1}
\end{eqnarray}
where approximation $a_1(x)$ is used 
in equation (\ref{landaupoi}) and 
$\alpha_2(x)$ is the approximation given in
equation~(\ref{pi_2_first_two_terms}).
With definitions (\ref{landauapprox})--(\ref{odd2lp1}), we can define the relative error by
\begin{equation}
\varepsilon_n(x)
=
\frac{|a_n(x) - \pi_2(x)|}{\pi_2(x)} \, .
\label{defvarepsk}
\end{equation}
This definition is consistent with~(\ref{epsevenll})
for even values of $n$. The results computed by
(\ref{defvarepsk}) are plotted in
Figure~\ref{figerror}(b) for $x=10^8$ and
$x=10^{10}$. 

In Figure~\ref{figerror}, we study the behaviour 
of the relative error~(\ref{defvarepsk}) for a fixed 
value of $x$. We observe that the relative error
$\varepsilon_n(x)$ is initially a decreasing 
sequence (for smaller values of~$n$). It reaches its
minimum and then it starts increasing again. Denoting
the value of $n$ where the relative error reaches its
minimum as $n_{\mathrm{min}}(x)$, we observe
that $n_{\mathrm{min}}(x)$ is an increasing 
function of $x$, at least for the three values
of $x$ considered in Figure~\ref{figerror}(a). 

Since $x$ is fixed in Figure~\ref{figerror}, 
the relative error~(\ref{defvarepsk})
is a constant multiple of the absolute error
$|a_n(x) - \pi_2(x)|$. In particular, the same conclusion
about $n_{\mathrm{min}}(x)$ could be reached when
considering the absolute errors instead of relative errors.
Since we plot results for two or three different (fixed)
values of $x$ in Figure~\ref{figerror}, we can also observe
that the relative error is (for many of the values
of $n$) a decreasing function of $x$, although this does
not look to be true for $n=2$ in
Figure~\ref{figerror}. 

To confirm this
observation, we plot the behaviour of the relative
error $\varepsilon_n(x)$ as a function of $x$ in 
Figure~\ref{figerrorasfunctionofx}. We use the first
eight values of $n$ and confirm that the relative
error $\varepsilon_n(x)$ resembles a decreasing 
function of $x$ for $n=1,$ $3,$ $4,$ \dots, $8,$ while
$\varepsilon_2(x)$ reaches its minimum between
$2 \times 10^5$ and $3 \times 10^5$. 
For smaller values of $x$, the approximation 
$a_2(x)$ overestimates $\pi_2(x)$, while it is 
an underestimate for larger values of $x$.
The graphs of $a_2(x)$ and $\pi_2(x)$ cross 
for the values of $x$ between 
$2 \times 10^5$ and $3 \times 10^5$, so in
this interval we can get very close to the 
correct answer, which results in the minimum 
of $\varepsilon_2(x)$ in
Figure~\ref{figerrorasfunctionofx}(a), because
our error definition~(\ref{defvarepsk}) includes
the absolute value. In Figure~\ref{figerrorasfunctionofx},
we also observe that the other errors are not
strictly decreasing, but they fluctuate with 
a decreasing trend, see, for example, the plot 
of $\varepsilon_8(x)$ in
Figure~\ref{figerrorasfunctionofx}(b).

\begin{figure}
\leftline{\hskip 2mm (a) \hskip 6.95cm (b)}
\centerline{{\hskip 4mm}
\includegraphics[height=5.5cm]{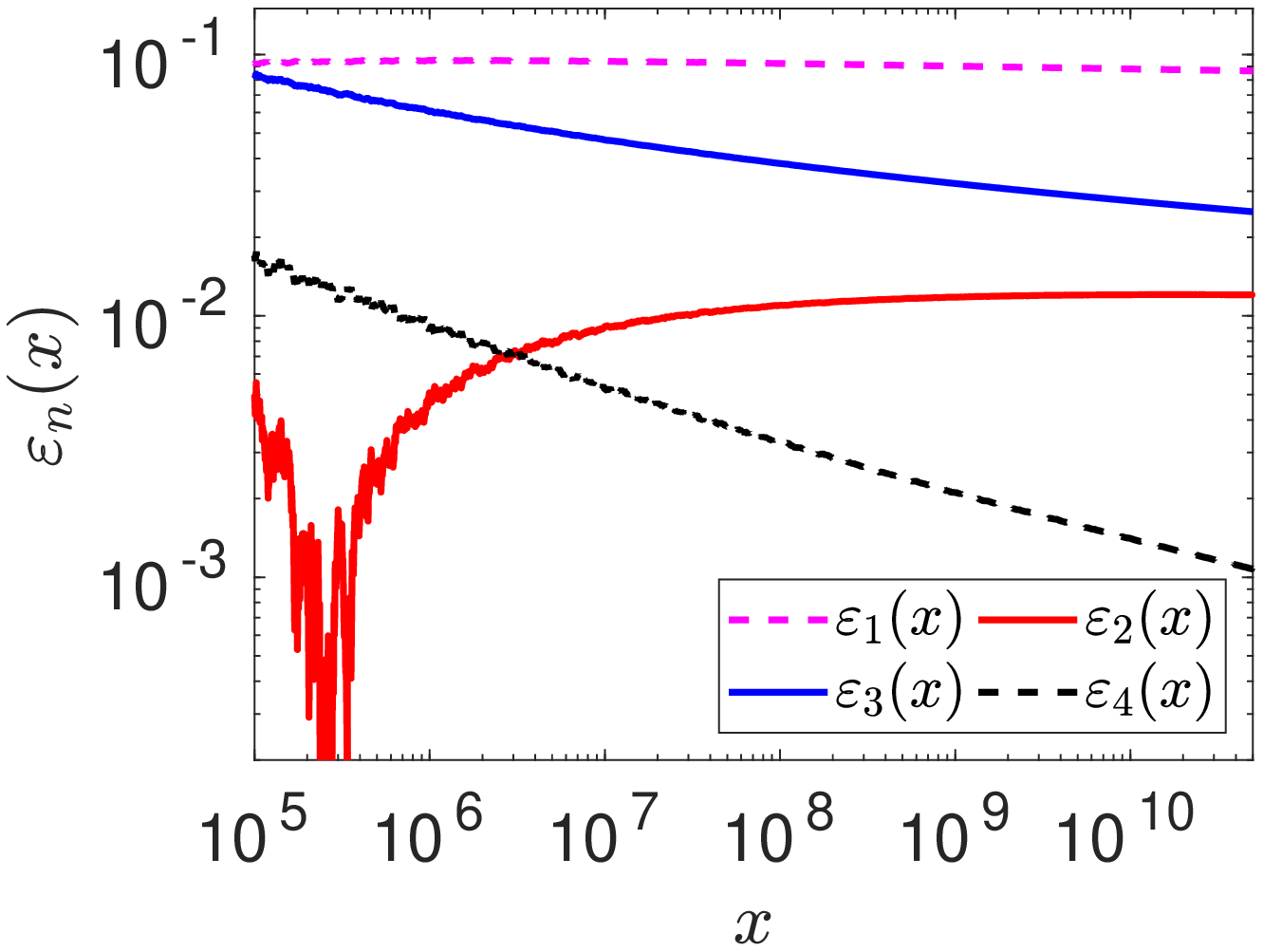}
\includegraphics[height=5.5cm]{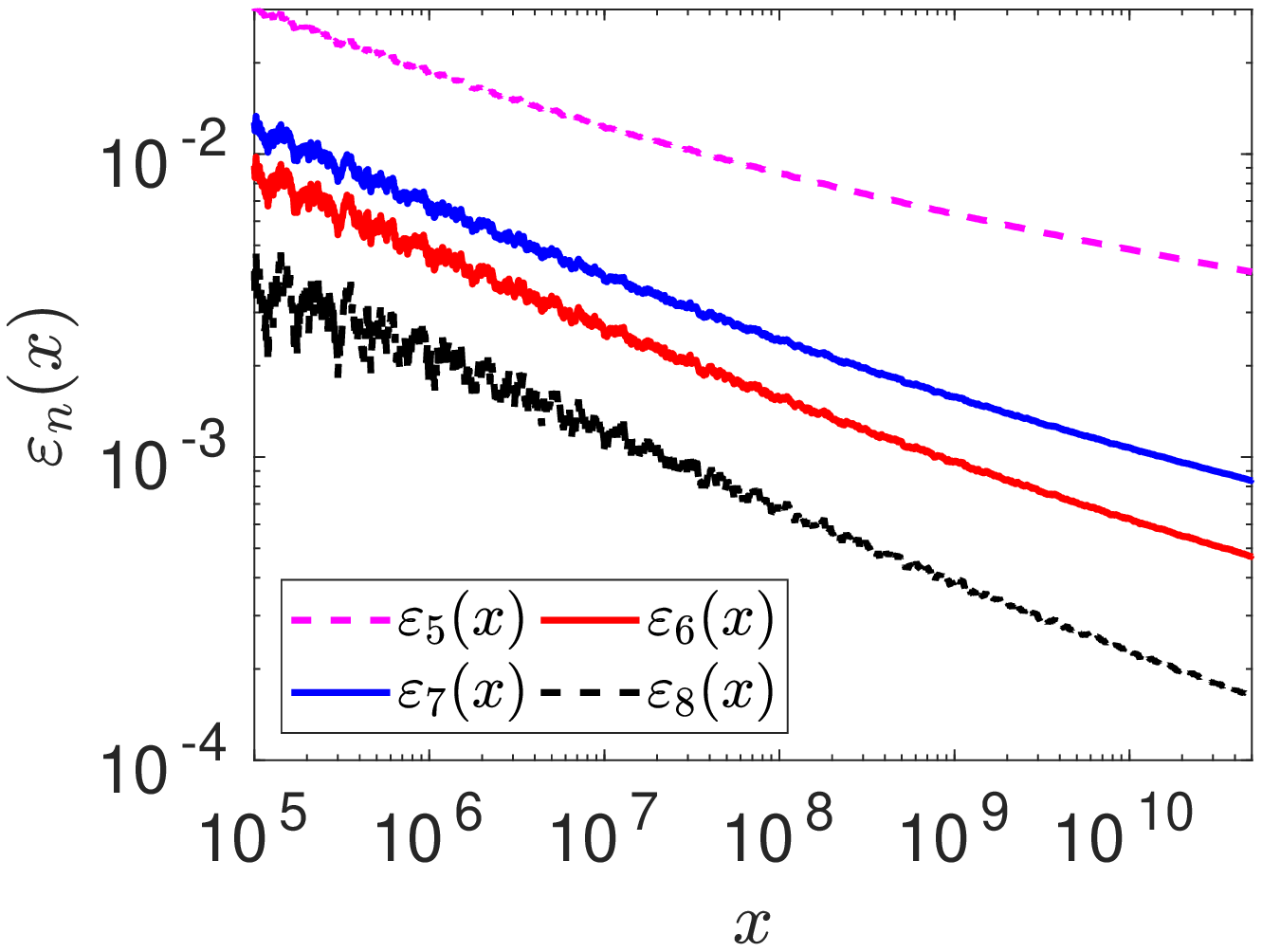}}
\caption{(a) 
{\it Relative errors $\varepsilon_n(x)$, for
$n=1$, $2$, $3$ and $4$, calculated
by~$(\ref{defvarepsk})$ as a function of $x$.
The relative error of Landau's
approximation~$(\ref{landauapprox})$ 
is plotted as $\varepsilon_1(x)$. The relative 
error of the approximation given in 
Theorem~$\ref{theoremM}$ is plotted 
as $\varepsilon_2(x)$.}
\hfill \break
(b) {\it Relative errors 
$\varepsilon_n(x)$, for
$n=5$, $6$, $7$ and $8$ as a function of $x$.
}
}
\label{figerrorasfunctionofx}
\end{figure}

\section{Discussion}

\label{secdiscussion}

\noindent In this paper, we have studied the behaviour of the semiprime
counting function $\pi_2(x),$ which is a special case ($k=2$) 
of the $k$-almost prime counting function $\pi_k(x)$. To generalize the 
presented results to the case $k \ge 3,$ we need to first generalize the
counting Lemma~\ref{counting_pi2}. Using the inclusion-exclusion
principle, it is possible to deduce the following counting formula
\begin{equation}
\label{count_form_pi_k}
\pi_k(x)
=\sum_{i=1}^{k} 
(-1)^{i-1}
\!\!\!
\sum_{p_1<p_2<\ldots<p_i\leq \sqrt[k]{x}} \!\!\!
\pi_{k-i} \! \left( \frac{x}{p_1 p_2 \ldots p_i} \right)\, ,
\end{equation}
where we define function $\pi_0(x)$ to be identically equal 
to 1, i.e. $\pi_0(x)=1$, and the sum over 
$p_1<p_2<\ldots<p_i\leq \sqrt[k]{x}$ means that we are 
summing $i$-times over all primes satisfying the given condition.
Substituting $k=2$ into equation~(\ref{count_form_pi_k}), we
obtain 
$$
\pi_2(x)
=
\sum_{p_1 \leq \sqrt{x}} \!\!\!
\pi_{1} \! \left( \frac{x}{p_1} \right)
\, - \!
\sum_{p_1<p_2\leq \sqrt{x}} \!\!\!
\pi_{0} \! \left( \frac{x}{p_1 p_2} \right).
$$
Using $\pi_0(x)=1$, we deduce equation~(\ref{countformpi2}). Thus,
equation~(\ref{count_form_pi_k}) provides a generalization of
equation~(\ref{countformpi2}), which expresses the $k$-almost
prime counting function $\pi_k(x)$ in terms of the counting functions
$\pi_1(x),$ $\pi_2(x)$, \dots, $\pi_{k-1}(x)$. It can be inductively
used to derive forms of coefficients of polynomials $P_{n,k}$ in 
the asymptotic series~(\ref{delange}). In addition to constants
$B_n$ and $C_n$, certain new constants will appear in such
calculations, including the (converging) sums of the form
$\sum_{p} (\log p)^i p^{-\ell}$ with $\ell\geq 2$ and 
$i \in {\mathbb N}$. For a detailed discussion of the 
asymptotic behaviour of these sums for $\ell = 1$, 
see Axler~\cite{Axler:2019:FFD}. Substituting $n=\pi(x)$ 
in~\cite[Theorem 5]{Axler:2019:FFD} gives a different 
expansion for the sums in~(\ref{landaulog}), which may be 
further examined using the prime number theorem.

There are, also, other possible approximations for $\pi_k(x)$. 
For example, Erd\H{o}s and S\'ark\H{o}zy~\cite{Erdos:1980:NPF} 
prove that
\begin{equation*}
\pi_k(x) 
<
\begin{cases}
\displaystyle c(\delta) \,
\frac{x}{\log x} 
\, \frac{(\log (\log x))^{k-1}}{(k-1)!} \,, 
&\text{for }1\leq k \leq (2-\delta)\log (\log x);
\\
\displaystyle
\rule{0pt}{5mm} c \, k^4 \, 2^{-k} \,
x \, \log x \, , &\text{for }k\geq 1;
\end{cases}
\end{equation*}
for some constants $c(\delta)$ and $c$. Other approximations, 
relating the function $\pi_k$ to some other products over primes 
are possible to obtain, as explained in \cite{Tenenbaum:2015:IAP}.

Functions $\pi_k(x)$ and $\Omega(x)$, used in 
expansion~(\ref{defOmegaLarge}), count the prime divisors 
with their multiplicity. Another possible generalization
is to investigate the related functions $N_k(x)$ and $\omega(x)$,
counting prime divisors without multiplicity. That is,
functions $N_k(x)$ and $\omega(x)$ are defined to be the number 
of natural numbers $n\leq x$ which have exactly $k$ distinct 
prime divisors and the number of distinct prime divisors of 
$x$, respectively. Finch~\cite[page 26]{Finch:2018:MC2} shows that 
\begin{equation*}
\frac{1}{x} \sum_{m \le x}
\omega(m)
\, \sim \,
\log (\log x)
+
0.2614972128\dots
+
\sum_{n=1}^\infty
\left(
- 1
+
\sum_{i=0}^{n-1}
\frac{\gamma_i}{i!}
\right)
\frac{(n-1)!}{\log^n x} \, ,
\end{equation*}
which has the higher order terms in the same form as in the
expansion~(\ref{defOmegaLarge}). Using the prime number theorem,
we also observe that
$
N_1(x) 
= 
\pi\left(x\right) + \pi\left(\sqrt{x}\right) 
+ 
\pi\left(\sqrt[3]{x}\right) + \ldots \sim \mbox{li}(x)
$ 
admits an identical asymptotic expansion as $\pi(x)$.
Delange~\cite[Theorem 1]{Delange:1971:FAS} and 
Tenenbaum \cite{Tenenbaum:2015:IAP} obtained the asymptotic
expansion of $N_k(x)$ in the form
\begin{equation*}
N_k(x) 
\, \sim \, 
\frac{x}{\log x}
\,
\sum_{n=0}^{\infty}
\frac{Q_{n,k} (\log (\log x))}{(\log x)^n} \, ,
\end{equation*}
where $Q_{n,k}$ are polynomials of degree $k-1$. Here, the expansion 
is similar to the expansion~(\ref{delange})
for $\pi_k(x)$, but the polynomials $P_{n,k}$ and $Q_{n,k}$
are different. Results about the leading terms of polynomials
$Q_{n,k}$ and about $Q_{0,k}$ have also been obtained, 
as in the case of $\pi_k$. Several different approximations 
for $N_k(x)$ are also possible to derive, as shown
in Tenenbaum~\cite{Tenenbaum:2015:IAP}, who points out 
that the function $N_k$ is easier to analyse than $\pi_k$, 
for larger values of $k$, relative to $\log (\log x)$. 
For example, the following holds uniformly for $x\geq 3$ and 
$(2+\delta)\log(\log x) \leq k \leq A \log(\log x) $:
$$
N_k(x) 
= 
C 
\, \frac{x \log x}{2^k} 
\left( 1 + 
\bigO
\! \left( (\log x)^{-\delta^2/5} \right) 
\right) \, ,
$$
where $A>0$, $0<\delta<1$ and $C \approx 0.378694$. 
Similar results, but for larger values of $k$, can 
be obtained for functions $\pi_k$ as 
well~\cite{Nicolas:1984:DNE}.

\begin{acknowledgements} \noindent
Authors would like to thank Julia Stadlmann for helpful comments during the
preparation of this manuscript.
\end{acknowledgements}

\affiliationone{%
   Drago\cb{s} Cri\cb{s}an\\
   Merton College\\ 
   Merton Street\\ 
   Oxford OX1 4JD\\
   United Kingdom
   \email{dragos.crisan@merton.ox.ac.uk}}
\affiliationtwo{%
   Radek Erban\\
   Mathematical Institute\\ 
   University of Oxford\\
   Woodstock Road\\
   Oxford OX2 6GG\\
   United Kingdom
   \email{erban@maths.ox.ac.uk}}
\end{document}